\documentclass[reqno]{amsart}
\usepackage[utf8]{inputenc} 
\usepackage{amssymb}
\usepackage{graphicx}
\usepackage{latexsym}
\usepackage{stmaryrd}
\usepackage{enumitem}
\usepackage[bookmarks,hyperfootnotes=false]{hyperref}
\usepackage{multirow}
\usepackage{xcolor}
\usepackage{caption}
\colorlet{darkishRed}{red!80!black}
\colorlet{darkishBlue}{blue!60!black}
\colorlet{darkishGreen}{green!60!black}
\usepackage{hyperref}
\hypersetup{
pdftitle = {End-faithful spanning trees in graphs without normal spanning trees},
pdfsubject = {End-faithful spanning trees in graphs without normal spanning trees},
pdfauthor = {Carl Bürger and Jan Kurkofka},
pdfkeywords = {},
draft = false,
bookmarksopen=true}
\hypersetup{
    colorlinks,
    linkcolor={red!60!black},
    citecolor={green!60!black},
    urlcolor={blue!60!black}
}

\usepackage{theoremref} 
\usepackage{tikz}
\usepackage{tikz-cd}
\usetikzlibrary{calc,through,intersections,arrows, trees, positioning, decorations.pathmorphing, cd}
\usepackage{relsize}
\usepackage{comment}
\usepackage{xifthen}
\usepackage{mathrsfs}
\usepackage{svg}
\usepackage{mathtools}
\usepackage{nccmath}
\usepackage{pifont}

\usepackage{geometry}
\newcommand{\SCnumberCite}[1]
{%
\ifthenelse{\equal{#1}{1}}{\cite{StarComb1StarsAndCombs}}{}%
\ifthenelse{\equal{#1}{2}}{\cite{StarComb2TheDominatedComb}}{}%
\ifthenelse{\equal{#1}{3}}{\cite{StarComb3TheUndominatedComb}}{}%
\ifthenelse{\equal{#1}{4}}{\cite{StarComb4TheUndominatingStar}}{}%
}

\newcommand{\SCnumberHand}[2]
{%
\ifthenelse{\equal{#1}{#2}}{\ding{43}\,}{}%
}

\newcommand{\SCintroList}[2]
{%
    \ifthenelse{\equal{#1}{#2}}{(this paper)}{\SCnumberCite{#1}}%
}

\newcommand{\SCintroDetermined}[1]
{%
    \ifthenelse{\equal{#1}{1}}{In this paper, we determine}{In the first paper of this series, we determined}%
}

\newcounter{quotecount}

\makeatletter
\newcommand*{\addFileDependency}[1]{
  \typeout{(#1)}
  \@addtofilelist{#1}
  \IfFileExists{#1}{}{\typeout{No file #1.}}
}
\makeatother

\newcommand{\Abs}[1]{\partial_{\Omega} {#1}}

\newcommand{\rest}{\upharpoonright}

\newcommand{\im}{\normalfont\text{im}}

\renewcommand{\subset}{\subseteq}
\renewcommand{\supset}{\supseteq}

\newcommand{\at}{attached to }

\newcommand{\parentheses}[1]{{\left( {#1} \right)}}

\newcommand{\of}{\parentheses}

\newcommand{\dc}[1]{\lceil #1\rceil}

\newcommand{\nt}{T_{\textup{\textsc{nt}}}}

\makeatletter

\def\calCommandfactory#1{%
   \expandafter\def\csname c#1\endcsname{\mathcal{#1}}}
\def\frakCommandfactory#1{%
   \expandafter\def\csname frak#1\endcsname{\mathfrak{#1}}}

\newcounter{ctr}
\loop
  \stepcounter{ctr}
  \edef\X{\@Alph\c@ctr}
  \expandafter\calCommandfactory\X
  \expandafter\frakCommandfactory\X
  \edef\Y{\@alph\c@ctr}
  \expandafter\frakCommandfactory\Y
\ifnum\thectr<26
\repeat

\setenumerate{label={\normalfont (\roman*)},itemsep=0pt}

\def\lowfwd #1#2#3{{\mathop{\kern0pt #1}\limits^{\kern#2pt\raise.#3ex
\vbox to 0pt{\hbox{$\scriptscriptstyle\rightarrow$}\vss}}}}
\def\lowbkwd #1#2#3{{\mathop{\kern0pt #1}\limits^{\kern#2pt\raise.#3ex
\vbox to 0pt{\hbox{$\scriptscriptstyle\leftarrow$}\vss}}}}

\def\ve{\kern-1.5pt\lowfwd e{1.5}2\kern-1pt}
\def\ev{\kern-1pt\lowbkwd e{0.5}2\kern-1pt}
\def\vf{\kern-2pt\lowfwd f{2.5}2\kern-1pt}

\newtheorem{theorem}{Theorem}[section] 

\newtheorem{lemma}[theorem]{Lemma}

\newtheorem{problem}[theorem]{Problem}

\newtheorem{innercustomthm}{Theorem}

\newenvironment{customthm}[1]
  {\renewcommand\theinnercustomthm{#1}\innercustomthm}
  {\endinnercustomthm}

\theoremstyle{definition}

\theoremstyle{remark}

\newcommand{\ntrank}{normal rank}

\newcommand{\rlt}{T_{\textsc{rl}}}
\newcommand{\ntr}{normally traceable}
\newcommand{\itr}{$\cI$-traceable}
\newcommand{\itrank}{$\cI$-rank}

\begin{document}
\title[End-faithful spanning trees in graphs without normal spanning trees]{End-faithful spanning trees in graphs\\without normal spanning trees}

\author{Carl Bürger}
\author{Jan Kurkofka}
\address{\hspace*{-\parindent}University of Hamburg, Department of Mathematics, Bundesstraße 55 (Geomatikum), 20146 Hamburg, Germany}
\email{carl.buerger@uni-hamburg.de, jan.kurkofka@uni-hamburg.de}


\@namedef{subjclassname@2020}{\textup{2020} Mathematics Subject Classification}

\begin{abstract}
Schmidt characterised the class of rayless graphs by an ordinal rank function, which makes it possible to prove statements about rayless graphs by transfinite induction.
Halin asked whether Schmidt's rank function can be generalised to characterise other important classes of graphs. 
We answer Halin's question in the affirmative. 


Another largely open problem raised by Halin asks for a characterisation of the class of graphs with an end-faithful spanning tree. 
A well-studied subclass is formed by the graphs with a normal spanning tree.
We determine a larger subclass,
the class of \ntr\  graphs, which consists of the connected graphs with a rayless tree-decomposition into normally spanned parts.
Investigating the class of \ntr\ graphs further we prove that, for every \ntr\ graph, having a rayless spanning tree is equivalent to all its ends being dominated. 
Our proofs rely on a characterisation of the class of \ntr\ graphs by an ordinal rank function that we provide. 
\end{abstract}

\vspace*{-1.2cm}
\maketitle


\vspace*{-.8cm}
\section{Introduction}
\noindent Schmidt~\cite{DiestelBook5,Schmidt1983} characterised the class of rayless graphs by an ordinal rank function, which makes it possible to prove statements about rayless graphs by transfinite induction.
For example, Bruhn, Diestel, Georgakopoulos and Sprüssel~\cite{UnfriendlyPartition,DiestelBook5} proved the unfriendly partition conjecture for the class of rayless graphs in this way.
At the turn of the millennium, Halin~\cite{halin2000} asked in his legacy collection of problems whether Schmidt's rank can be generalised to characterise other important classes of graphs besides the class of
rayless graphs. 
In this paper we answer Halin's question in the affirmative: we characterise an important class of graphs by an ordinal rank function.



Our first main result addresses another largely open problem raised by Halin.
Call a spanning tree $T$ of a graph $G$ \emph{end-faithful} if the natural map \mbox{$\varphi\colon\Omega(T)\to\Omega(G)$} satisfying $\omega\subset\varphi(\omega)$ is bijective. Here, $\Omega(T)$ and $\Omega(G)$ denote the set of ends of $T$ and of $G$, respectively.
Halin~\cite{halin64} conjectured that every connected graph has an end-faithful spanning tree.
However, Seymour and Thomas~\cite{endfaithfulCounterexample} and Tho\-mas\-sen~\cite{ThomassenEndfaithfulCounterexample} constructed uncountable counterexamples; for instance, there exists a connected graph that has precisely one end but all whose spanning trees must contain a subdivision of $T_{\aleph_1}$ (recall that $T_\kappa$ denotes the $\kappa$-branching tree, for a cardinal $\kappa$).
Ever since, it has been an open problem to characterise the class of graphs that admit an end-faithful spanning tree.

Normal spanning trees are important examples of end-faithful spanning trees.
Given a graph $G$, a rooted tree $T\subset G$ is \emph{normal} in $G$ if the endvertices of every $T$--$T$ path in $G$ are comparable in the tree-order of~$T$, cf.~\cite{DiestelBook5}.
Call a set $U$ of vertices of a graph~$G$ \emph{normally spanned} in $G$ if $U$ is contained in a tree $T\subset G$ that is normal in~$G$. 
The graph $G$ is \emph{normally spanned} if $V(G)$ is normally spanned in~$G$, i.e., if $G$ has a normal spanning tree. 
Thus, being normally spanned is a first sufficient condition for the existence of an end-faithful spanning tree, and the normally spanned graphs are well understood: they are characterised by Jung's Theorem~\ref{thm: jung normal tree char}, for example.

A second existence result for end-faithful spanning trees is due to Polat~\cite{PolatEndFaithfulT1Free} and directly addresses the counterexamples by Seymour and Thomas and by Thomassen:  every connected graph that does not contain a subdivision of $T_{\aleph_1}$ has an end-faithful spanning tree.

As our first main result, we determine a new subclass of the class of graphs with an end-faithful spanning tree.
Call a connected graph $G$ \emph{\ntr} if it has a rayless tree-decomposition into parts that are normally spanned in $G$.
For the definition of tree-decompositions see~\cite{DiestelBook5}.

\begin{customthm}{\ref{thm: ntrank implies endfaithful st}}
Every \ntr\ graph has an end-faithful spanning tree.
\end{customthm}

\noindent Our theorem easily extends the two known results regarding sufficient conditions for the existence of end-faithful spanning trees:
On the one hand, every normally spanned graph has a trivial tree-decomposition into one normally spanned part.
On the other hand, every connected graph without a subdivision of $T_{\aleph_1}$ has a rayless tree-decomposition into countable parts by a result of Seymour and Thomas~\cite{ExcludingInfiniteTrees}, Theorem~\ref{thm: Tk}, and countable vertex sets are normally spanned by Jung's Theorem~\ref{thm: jung normal tree char}.

In both cases, the extension is proper:
The \mbox{$\aleph_1$-\emph{branching}} \emph{trees with tops} are the graphs obtained from the rooted $T_{\aleph_1}$ by selecting uncountably many rooted rays and adding for every selected ray $R$ a new vertex, its \emph{top}, and joining it to infinitely many vertices of~$R$~\cite{DiestelLeaderNST}.
Every $T_{\aleph_1}$ with tops has a star-decomposition into normally spanned parts where $T_{\aleph_1}$ forms the central part and each top plus its neighbours forms a leaf's part.
However, not every $T_{\aleph_1}$ with tops has a normal spanning tree~\cite{DiestelLeaderNST,PitzNewNSTobstructions}, and every $T_{\aleph_1}$ with tops contains $T_{\aleph_1}$ as a subgraph.



As our second main result, we extend 
two results on rayless spanning trees.
Recall that a vertex $v$ of a graph $G$ \emph{dominates} a ray $R\subset G$ if there is an infinite $v$--$R$ fan in~$G$. 
An end of $G$ is \emph{dominated} if one (equivalently:~each) of its rays is dominated, see~\cite{DiestelBook5}.
For a connected graph $G$, having a rayless spanning tree is equivalent to all the ends of $G$ being dominated if $G$ is normally spanned~\cite{StarComb3TheUndominatedComb} or if $G$ does not contain a subdivision of $T_{\aleph_1}$~\cite{PolatEndFaithfulT1Free}.
Our second main result extends these results, and any $T_{\aleph_1}$ with all tops witnesses that this extension is proper.

\begin{customthm}{\ref{thm: normally traceable rayless}}
For every \ntr\ graph $G$, having a rayless spanning tree is equivalent to all the ends of $G$ being dominated.
\end{customthm}

Finally, as our third main result we characterise the class of \ntr\ graphs by an ordinal rank function that we call the normal rank:

\begin{customthm}{3}\label{main result: rayless tree}
For every graph $G$ the following assertions are equivalent:
\begin{enumerate}
    \item $G$ is \ntr ;
    \item $G$ has a \ntrank.
\end{enumerate}
\end{customthm}
\noindent We use this in the proofs of all our results on \ntr\ graphs.

This paper is organised as follows.
Section~\ref{section: tools and terminology} provides the tools and terminology that we use throughout this paper.
In Section~\ref{section: ranks and ideals} we show how ideals can be used to define ordinal rank functions, and we structurally characterise the arising classes of graphs.
Then, in Section~\ref{section: rayless tree-decompositions} we introduce the \ntrank\ and prove Theorem~\ref{main result: rayless tree}.
We prove Theorem~\ref{thm: ntrank implies endfaithful st} in Section~\ref{section: end-faithful spanning trees} and we prove Theorem~\ref{thm: normally traceable rayless} in Section~\ref{section: rayless spanning trees}.

\section{Tools and terminology}\label{section: tools and terminology}
\noindent Any graph-theoretic notation not explained here can be found in Diestel's textbook~\cite{DiestelBook5}.

Recall that a \emph{comb} is the union of a ray $R$ (the comb's \emph{spine}) with infinitely many disjoint finite paths, possibly trivial, that have precisely their first vertex on~$R$. 
The last vertices of those paths are the \emph{teeth} of this comb.
Given a vertex set $U$, a \emph{comb attached to} $U$ is a comb with all its teeth in $U$, and a \emph{star attached to}~$U$ is a subdivided infinite star with all its leaves in $U$.

\begin{lemma}[Star-Comb Lemma~{\cite[8.2.2]{DiestelBook5}}]
Let $G$ be any connected graph and let $U\subseteq V(G)$ be infinite.
Then $G$ contains either a comb 
\at $U$ or a star \at $U$.
\end{lemma}

A \emph{ray} is a one-way infinite path.
An \emph{end} of a graph $G$, as defined by Halin~\cite{halin64}, is an equivalence class of rays of~$G$.
Here, two rays of~$G$ are said to be \emph{equivalent} if for every finite vertex set $X\subseteq V(G)$ both have a subray (also called \emph{tail}) in the same component of $G-X$. 
So in particular every end $\omega$ of $G$ chooses, for every finite vertex set $X\subset V(G)$, a unique component $C(X,\omega)$ of $G-X$ in which every ray of $\omega$ has a tail. 
In this situation, the end $\omega$ is said to \emph{live} in $C(X,\omega)$.
The set of ends of a graph $G$ is denoted by~$\Omega(G)$.
We use the convention that $\Omega$ always denotes the set of ends $\Omega(G)$ of the graph named $G$.

Let us say that an end $\omega$ of a graph $G$ is contained \emph{in the closure} of $M$, where $M$ is either  a subgraph of $G$ or a set of vertices of $G$, if for every finite vertex set $X\subseteq V(G)$ the component $C(X,\omega)$ meets $M$.
Equivalently, $\omega$ lies in the closure of $M$ if and only if $G$ contains a comb \at $M$ with its spine in~$\omega$.
We write $\Abs{M}$ 
for the subset of $\Omega$ that consists of the ends of $G$ lying in the closure of $M$.

A subset $X$ of the ground set of a poset $(P,{\le})$ is \emph{cofinal} in $P$, and ${\le}$, if for every $p\in P$ there is an $x\in X$ with $x\ge p$. We say that a rooted tree $T\subset G$ contains a set $U$ \emph{cofinally} if $U\subset V(T)$ and $U$ is cofinal in the tree-order of $T$. We remark that the original statement of the following lemma also takes critical vertex sets in the closure of $T$ or $U$ into account.

\begin{lemma}[{\cite[Lemma~2.13]{StarComb1StarsAndCombs}}]\label{lemma: closure and tree containing U cofinally}
Let $G$ be any graph.
If $T\subset G$ is a rooted tree that contains a vertex set~$U$ cofinally, then $\Abs{T}=\Abs{U}$.
\end{lemma}

Suppose that $H$ is any subgraph of $G$ and $\varphi\colon\Omega(H)\to\Omega(G)$ is the natural map satisfying $\eta\subset \varphi(\eta)$ for every end $\eta$ of $H$.
Furthermore, suppose that a set $\Psi\subset\Omega(G)$ of ends of $G$ is given.
We say that $H$ is \emph{end-faithful} for $\Psi$ if $\varphi\rest \varphi^{-1}(\Psi)$ is injective and $\im(\varphi)\supset\Psi$. And $H$ \emph{reflects} $\Psi$ if $\varphi$ is injective with $\im(\varphi)=\Psi$.
A spanning tree of $G$ that is end-faithful for all the ends of $G$ is \emph{end-faithful}.

\begin{lemma}[{\cite[Lemma~2.11]{StarComb1StarsAndCombs}}]\label{lemma: normal tree reflects the ends}
If $G$ is any graph and $T\subset G$ is any normal tree, then $T$ reflects the ends of $G$ in the closure of $T$.
\end{lemma}

Given any graph $G$, a set $U\subseteq V(G)$ of vertices is \emph{dispersed} in $G$ if there is no end in the closure of $U$ in $G$. Equivalently, $U$ is dispersed if and only if $G$ contains no comb \at ~$U$. 
In~\cite{jung69}, Jung proved that normally spanned sets of vertices can be characterised in terms of dispersed vertex sets:

\begin{theorem}[{Jung~\cite[Satz~6]{jung69};~\cite[Theorem~3.5]{StarComb1StarsAndCombs}}]\label{thm: jung normal tree char}
Let $G$ be any graph. A vertex set $U\subset V(G)$ is normally spanned in $G$ if and only if it is a countable union of dispersed sets. In particular, 
$G$ is normally spanned if and only if $V(G)$ is a countable union of dispersed sets.
\end{theorem}


\section{Ideal rank}\label{section: ranks and ideals}

\noindent There is more than one rank function that generalises Schmidt's rank.
Before we focus on the \ntrank\ in the next sections, here we present a scheme that captures the essential idea behind Schmidt's rank, and we show how this scheme translates to structural characterisations of the classes of graphs defined by the rank functions which follow this scheme.

Recall that an \emph{ideal} is a class that contains the empty set, is closed under finite unions and is closed under taking subsets of its elements.
Thus, it is the dual notion of a filter.
We remark that usually, ideals are required to be sets, but in this paper we allow them to be classes.

Suppose that $\cI$ is any ideal.
Let us assign $\cI$-\emph{rank} 0 to all the graphs whose vertex sets are contained in~$\cI$.
Given an ordinal $\alpha>0$, we assign $\cI$-\emph{rank} $\alpha$ to every graph $G$ that does not already have an $\cI$-rank less than~$\alpha$ and which has a set $X$ of vertices with $X\in\cI$ such that every component of $G-X$ has some $\cI$-rank less than~$\alpha$.
If $\cI$ is the class of all finite sets, for example, then the $\cI$-rank coincides with Schmidt's rank~\cite{DiestelBook5,Schmidt1983}.
From now on, we refer to Schmidt's rank as $\aleph_0$-\emph{rank}.

The structure of the graphs with an \itrank\ for some pre-specified ideal $\cI$ can be described in terms of tree-decompositions and~$\cI$, as follows.
Let us say that a graph $G$ is $\cI$-\emph{traceable} if it has a rayless tree-decomposition with all parts in~$\cI$.

\begin{theorem}\label{raylessTDC}
For every graph $G$ and every ideal~$\cI$ the following assertions are equivalent:
\begin{enumerate}
    \item $G$ has an $\cI$-rank;
        \item $G$ is \itr .
\end{enumerate}
Moreover, if \emph{(i)} and \emph{(ii)} hold, then
\[
\cI\text{-rank of }G=\min\big\{\,\aleph_0\text{-rank of }T \;\big\vert\; (T,\cV)\text{ is a rayless tree-decomposition of }G\text{ with all parts in }\cI\,\big\}.
\]
\end{theorem}

\begin{proof}
Let $G$ be any graph.
To show the equivalence (i)$\leftrightarrow$(ii) together with the `moreover' part of the theorem, it suffices to show the following two assertions:
\begin{enumerate}
    \item[(1)] If $G$ has a tree-decomposition witnessing that $G$ is \itr , then $G$ has an \itrank\ which is at most the $\aleph_0$-rank of the decomposition tree.
    \item[(2)] If $G$ has an \itrank , then $G$ is \itr\ and this is witnessed by a tree-decomposition whose decomposition tree has $\aleph_0$-rank at most the \itrank\ of $G$.
\end{enumerate}

\noindent (1) 
Let $(T,\cV)$ be any rayless tree-decomposition of $G$ with all parts in~$\cI$.
We show by induction on the $\aleph_0$-rank $\alpha$ of $T$ that $G$ has \itrank\ at most~$\alpha$.
If $\alpha=0$, then $G$ has \itrank ~$0$.
Otherwise $\alpha>0$.
Let $W\subset V(T)$ be any finite vertex set such that every component of $T-W$ has $\aleph_0$-rank less than~$\alpha$.
Then the finite union $X:=\bigcup_{t\in W}V_t$ is contained in~$\cI$.
Every component of $G-X$ is contained in $\bigcup_{t\in T'}G[V_t]$ for some component $T'$ of $T-W$, so by the induction hypothesis every component of $G-X$ has \itrank\ less than~$\alpha$.
Thus, $G$ has \itrank\ at most~$\alpha$.

(2) We prove the statement by induction on the $\cI$-rank $\alpha$ of~$G$.
If $\alpha=0$, then $V(G)\in\cI$ and the trivial tree-decomposition of $G$ into the single part $V(G)$ is as desired.
Otherwise $\alpha>0$.
Then there exists a vertex set $X\subset V(G)$ with $X\in\cI$ such that every component of $G-X$ has \itrank\ less than~$\alpha$.
By the induction hypothesis, every component $C$ of $G-X$ has a rayless tree-decomposition $(T_C,\cV_C)$ with $\cV_C=(\,V_C^t\mid t\in T_C\,)$ such that every part is contained in~$\cI$ and the $\aleph_0$-rank of $T_C$ is less than~$\alpha$.
Without loss of generality the trees $T_C$ are pairwise disjoint.
We choose from every tree $T_C$ an arbitrary node $t_C\in T_C$.
Then we let the tree $T$ be obtained from the disjoint union $\bigcup_C T_C$ by adding a new vertex $t_\ast$ that we join to all the chosen nodes~$t_C$.
We define the family $\cV=(\,V^t\mid t\in T\,)$ by letting $V^t:=V_C^t\cup X$ for all $t\in T_C$ and $V^{t_\ast}:=X$.
Then $(T,\cV)$ is a rayless tree-decomposition of $G$ with all parts in~$\cI$, and the $\aleph_0$-rank of $T$ is at most~$\alpha$ because every component of $T-t_\ast$ has $\aleph_0$-rank less than~$\alpha$.
\end{proof}

As an application of Theorem~\ref{raylessTDC}, we extend the following theorem by Seymour and Thomas:

\begin{theorem}[{\cite[Theorem~1.3]{ExcludingInfiniteTrees}}]\label{thm: Tk}
For every graph $G$ and every uncountable cardinal $\kappa$ the following assertions are equivalent:
\begin{enumerate}
    \item $G$ contains no $T_\kappa$ minor;
    \item $G$ has a rayless tree-decomposition into parts of size less than~$\kappa$.
\end{enumerate}
\end{theorem}
\noindent Indeed, applying Theorem~\ref{raylessTDC} to the ideal of all sets of size less than~$\kappa$ and calling the arising rank function the $\kappa$-\emph{rank}, we may add:
\begin{enumerate}
    \item[(iii)] $G$ has a $\kappa$-rank.
\end{enumerate}
We remark that, for regular uncountable cardinals $\kappa$, Seymour and Thomas also showed that a graph contains a $T_{\kappa}$ minor if and only if it contains a subdivision of~$T_{\kappa}$.

\section{Normally traceable graphs}\label{section: rayless tree-decompositions}


\noindent Let $G$ be any connected graph.
By Jung's Theorem~\ref{thm: jung normal tree char}, the vertex sets $X\subset V(G)$ that are normally spanned in~$G$ form an ideal which we denote by~$\cI(G)$.
We call the $\cI(G)$-rank of a subgraph $H\subset G$ the \emph{\ntrank } of~$H$ in~$G$.
The graph $G$ has  \emph{\ntrank } $\alpha$ for an ordinal $\alpha$ if $G$ has \ntrank\ $\alpha$ in $G$.

Since we have defined the normal rank using ideals, Theorem~\ref{raylessTDC} implies Theorem~\ref{main result: rayless tree}.
Below, we point out a few properties of the \ntrank.

\begin{lemma}
Let $G$ be any connected graph.
\begin{enumerate} 
    \item If $G$ has $\aleph_1$-rank $\alpha$, then $G$ has \ntrank\  at most~$\alpha$.
    \item There are graphs that have a \ntrank\ but that have neither an $\aleph_1$-rank nor a normal\\spanning tree.
\end{enumerate}
\end{lemma}

\begin{proof} 
(i) We show that every subgraph $H\subseteq G$ of $\aleph_1$-rank $\alpha$ has \ntrank\ at most~$\alpha$ in $G$, by induction on $\alpha$; for $H=G$ this establishes (i). 
The vertex set of any countable subgraph of $G$ is normally spanned in $G$ by Jung's Theorem~\ref{thm: jung normal tree char}, so the base case holds. For the induction step suppose that $\alpha>0$. We find a countable vertex set $X\subset V(H)$ so that every component of $H-X$ has some $\aleph_1$-rank less than~$\alpha$. As $X$ is countable it is also normally spanned in $G$. 
By the induction hypothesis every component of $H-X$ has \ntrank\ less than~$\alpha$ in $G$. Hence $X$ witnesses that $H$ has \ntrank\ at most~$\alpha$ in $G$.

(ii) Let $G$ be any $T_{\aleph_1}$ with all tops and all edges between each top and its corresponding ray. Then~$G$ has \ntrank\ $1$ because $G-T_{\aleph_1}$ consists only of isolated vertices. 
However, $G$ has no $\aleph_1$-rank by Theorem~\ref{thm: Tk}, and $G$ has no normal spanning tree as pointed out in~\cite{DiestelLeaderNST,PitzNewNSTobstructions}.
\end{proof}

\begin{lemma}\label{lemma:ntrankOfSubgraphs}
Let $H\subset H'\subset G$ be any three graphs with $G$ connected.
\begin{enumerate}
    \item If $H'$ has \ntrank\ $\alpha$ in $G$, then $H$ has \ntrank\ at most~$\alpha$ in $G$. 
\end{enumerate}
For the second item, $H'$ is required to be connected as well.
\begin{enumerate}[resume]
    \item If $H$ has \ntrank\ $\alpha$ in $G$, then $H$ has \ntrank\ at most~$\alpha$ in $H'$.\\
    In particular, if $H$ has \ntrank\ $\alpha$ in $G$ and $H$ is connected, then $H$\\has \ntrank\ at most~$\alpha$.
\end{enumerate} 
\end{lemma}
\begin{proof}
We prove~(i) by induction on~$\alpha$.
If $\alpha=0$, then the vertex set of $H'$ is normally spanned in $G$; in particular, the vertex set of $H\subset H'$ is normally spanned in $G$.

Otherwise $\alpha>0$.
Then there exists a vertex set $X\subset V(H')$ that is normally spanned in $G$ such that every component of $H'-X$ has \ntrank\ less than~$\alpha$ in $G$.
Every component of $H-X$ is contained in a component of $H'-X$ and hence has \ntrank\ less than~$\alpha$ in $G$ by the induction hypothesis.
Thus, $H$ has \ntrank\ at most~$\alpha$ in $G$.

We prove~(ii) by induction on $\alpha$ as well.
If $\alpha=0$, then the vertex set of $H$ is normally spanned in~$G$.
In particular, by Jung's Theorem~\ref{thm: jung normal tree char}, the vertex set of $H$ is normally spanned in $H'\subset G$, so $H$ has \ntrank\ $0$ in $H'$ as desired.

Otherwise $\alpha>0$.
Then there exists a vertex set $X\subset V(H)$ that is normally spanned in $G$ such that every component of $H-X$ has \ntrank\ less than~$\alpha$ in $G$. 
Note that $X$ is also normally spanned in $H'\subset G$ by Jung's Theorem~\ref{thm: jung normal tree char}.
By the induction hypothesis, every component of $H-X$ has \ntrank\ less than~$\alpha$ in~$H'$.
Thus, $H$ has \ntrank\ at most~$\alpha$ in $H'$.
\end{proof}

\section{End-faithful spanning trees}\label{section: end-faithful spanning trees}
\noindent In this section we prove that every \ntr\ graph has an end-faithful spanning tree.
Our proof requires some preparation.

\begin{lemma}\label{lemma: reflecting forest}
Let $G$ be any graph and let $\Psi\subseteq \Omega(G)$ be any set of ends of~$G$. 
Furthermore, let $H\subset G$ be any spanning forest that reflects $\Psi$ and let $C$ be any component of~$H$.
If a spanning tree $T$ of $G$ arises from $H$ by adding one $D$--$C$ edge for every component $D\neq C$ of $H$, then $T$ reflects~$\Psi$.\qed
\end{lemma}

\begin{lemma}\label{lemma: endfaithful and complete countable set}
Let $G$ be any graph with a spanning tree $T$ that reflects a set $\Psi\subseteq \Omega(G)$ and let $R\subset G$ be a ray from some end in $\Psi$. Then there exists a spanning tree $T'\subset G$ that reflects $\Psi$ and contains $R$.

Moreover, $T'$ can be chosen such that no end other than the end of $R$ lies in the closure of the symmetric difference $E(T)\triangle E(T')$ (viewed as a subgraph of $G$).
\end{lemma}
\noindent The `moreover' part of the lemma says that $T$ and $T'$ differ only locally. 
Note that there may also be no end in the closure of $E(T)\triangle E(T')$.

\begin{proof}
Given $T\subset G$, $\Psi$ and $R$ as in the statement of the lemma, we root $T$ arbitrarily and write $\omega$ for the end of $R$ in~$G$.
Furthermore, we write $R_T$ for the unique rooted ray in $T$ that is equivalent to $R$, and we pick a sequence $P_0,P_1,\ldots$ of pairwise disjoint $R$--$R_T$ paths in $G$.
We write $C$ for the comb $C:=R\cup\bigcup_{n} P_n$, and we write $U$ for the vertex set of the subtree $\dc{C}_T$ of $T$ induced by the down-closure of $V(C)$ in~$T$.
Note that $R_T\subset\dc{C}_T$ because the paths $P_0,P_1,\ldots$ meet $R_T$ infinitely often.
By standard arguments we have $\Abs{C}=\{\omega\}$, and so $\Abs{U}=\{\omega\}$ follows by Lemma~\ref{lemma: closure and tree containing U cofinally}.
Since $T$ reflects $\Psi$ and $\dc{C}_T$ contains only rays from $\omega$, we deduce that $\dc{C}_T$ is either rayless or one-ended.
As $\dc{C}_T$ contains the ray $R_T$, it is one-ended.

Next, we define an edge set $F\subset E(\dc{C}_T)$, as follows.
If $R$ has a tail in $R_T$, then we set $F=\emptyset$.
Otherwise $R$ has no tail in $R_T$.
Then we select infinitely many pairwise edge-disjoint $C$-paths $Q_0,Q_1,\ldots$ in the ray $R_T$ (these exist because $R$ has no tail in $R_T$).
We choose one edge of every path $Q_n$ and we let $F$ consist of all the chosen edges, completing the definition of $F$.

The graph $(\dc{C}_T\cup C)-F$ is a connected subgraph of $G$ and inside it, we extend $C$ arbitrarily to a spanning tree $T_R$.
Then $T_R$ has vertex set $U$, and $T_R$ reflects $\{\omega\}$:
Every ray $R'$ in $T_R$ that is disjoint from $R$ meets at most one component of $C-R$ because $C$ and $R'$ are contained in the tree $T_R$, and hence $R'$ must have a tail in $\dc{C}_T- C$.
But $\dc{C}_T$ contains just one rooted ray, namely the ray $R_T$, and either $R_T$ contains a tail of $R$ or $F$ consists of infinitely many edges of $R_T$, contradicting the existence of~$R'$ in $T_R\subset (\dc{C}_T\cup C)-F$.
It remains to extend $T_R$ to a spanning tree of $G$ reflecting $\Psi$.
For this, we consider the collection $\{\,T_i\mid i\in I\,\}$ of all the components of $T-U$.
By the choice of $U$, every end $\omega'$ of $G$ other than $\omega$ is still represented by an end of one of the trees $T_i$:
Indeed, if $\omega'$ is an end of $G$ other than $\omega$, then it does not lie in the closure of $U$, and hence every ray in $\omega'$ has a tail that avoids $U$.
In particular, every ray in $T$ that lies in $\omega'$ has some tail that avoids~$U$.
Therefore, the union of $T_R$ and all the trees $T_i$ is a spanning forest of $G$ reflecting~$\Psi$. 

We extend this spanning forest to a spanning tree $T'$ by adding all the $T_i$--$T_R$ edges of $T$ for every $i\in I$ (note that $T$ contains precisely one $T_i$--$T_R$ edge for every $i\in I$ as $T\cap G[U]=\dc{C}_T$ is connected).
Then $T'$ reflects $\Psi$ again by Lemma~\ref{lemma: reflecting forest}.
To see $\Abs{(E(T)\triangle E(T'))}\subset\{\omega\}$ recall $\Abs{G[U]}=\{\omega\}$ and note that the symmetric difference is contained in $G[U]$ entirely.
\end{proof}

\begin{lemma}\label{lemma: ends are contained in closure of X or C}
Let $G$ be any graph and let $X\subseteq V(G)$ be any vertex set. 
\begin{enumerate}
    \item Every end of $G$ is contained in the closure of $X$ in $G$ or in the closure of some component\\of $G-X$ in $G$.
    \item Every end of $G$ that is contained in the closure of two distinct components of $G-X$ in~$G$\\is also contained in the closure of $X$ in $G$. 
\end{enumerate}
\end{lemma}

\begin{proof}
(i) Let $\omega$ be any end of $G$ and let $R\in \omega$ be any ray. 
Then either the vertex set of $R$ intersects $X$ infinitely often, or $R$ has a tail that is contained in some component $C$ of $G-X$. 
In the first case, $\omega$ is contained in the closure of $X$, and in the second case it is contained in the closure of $C$ in $G$.

(ii) Let $C$ and $C'$ be two distinct components of $G-X$ and suppose that $\omega$ is any end of $G$ that is contained in the closure of both $C$ and $C'$ in $G$. 
If  $S\subseteq V(G)$ is any finite vertex set, then the component $C(S,\omega)$ meets both $C$ and $C'$.  As $X$ separates $C$ and $C'$ in $G$ it follows that $C(S,\omega)$ meets $X$ as well. We conclude that $\omega$ is contained in the closure of $X$ in $G$.  
\end{proof}

\begin{lemma}\label{lemma: IH}
Let $G$ be any connected graph, let $X\subseteq V(G)$ be normally spanned in~$G$ and let $C$ be any component of $G-X$ so that $G[C\cup X]$ is connected. If $C$ has \ntrank\ $\xi$ in~$G$, then $G[C\cup X]$ has \ntrank\ at most~$\xi$. 
\end{lemma}

\begin{proof}
Suppose that $C$ is a component of $G-X$ that has \ntrank\ $\xi$ in $G$. 
If~$\xi=0$, then $V(C)$ is normally spanned in $G$ and $G[C\cup X]$ has a normal spanning tree by Jung's Theorem~\ref{thm: jung normal tree char}, so $G[C\cup X]$ has \ntrank\ $0$ as desired.
Otherwise there is a vertex set $Y\subseteq V(C)$ that is normally spanned in $G$ and satisfies that every component of $C-Y$ has \ntrank\ less than~$\xi$ in $G$. Note that $X\cup Y$ is normally spanned in $G$ by Jung's Theorem~\ref{thm: jung normal tree char}. Therefore $X\cup Y$ witnesses that $G[C\cup X]$ has \ntrank\ at most~$\xi$ in $G$.
Finally, Lemma~\ref{lemma:ntrankOfSubgraphs}~(ii) implies that $G[C\cup X]$ has \ntrank\ at most~$\xi$.
\end{proof}

\begin{customthm}{1}\label{thm: ntrank implies endfaithful st}
Every \ntr\ graph has an end-faithful spanning tree.
\end{customthm}

\begin{proof} By Theorem~\ref{main result: rayless tree} we may prove the statement via induction on the \ntrank\ of $G$. If $G$ has \ntrank\ $0$, then it has a normal spanning tree, and normal spanning trees are end-faithful.
For the induction step suppose that $G$ has \ntrank\ $\alpha>0$, and let $X\subset V(G)$ be any vertex set that is normally spanned in $G$ and satisfies that every component of $G-X$ has \ntrank\ less than~$\alpha$ in~$G$.
By replacing $X$ with the vertex set of any normal tree in $G$ that contains $X$, we may assume that $X$ is the vertex set of a normal tree $\nt\subset G$; indeed, every component of $G-X$ still has \ntrank\ less than~$\alpha$ in $G$ by Lemma~\ref{lemma:ntrankOfSubgraphs}~(i). 
 Note that, by Lemma~\ref{lemma: normal tree reflects the ends}, the tree $\nt$ reflects the ends of $G$ in the closure of $X$.

By Lemma~\ref{lemma: ends are contained in closure of X or C}~(i), every end of $G$ is contained in the closure of $X$ in $G$ or in the closure of some component of $G-X$. And by Lemma~\ref{lemma: ends are contained in closure of X or C}~(ii), every end of $G$ that is contained in the closure of two distinct components of $G-X$ in $G$ is also contained in the closure of $X$ in $G$. 
Thus, by Lemma~\ref{lemma: reflecting forest} it suffices to find in each component $C$ of $G-X$ a spanning forest $H_C$ so that every component of $H_C$ sends an edge in $G$ to $\nt$ and so that $H_C$ reflects $\Abs{C}\setminus\Abs{X}$.

For this, consider any component $C$ of $G-X$. 
Let $P$ be the (possibly one-way infinite) path in $\nt$ that is formed by the down-closure of $N(C)$ in $\nt$.
Then by Lemma~\ref{lemma: IH} the graph $G[C\cup P]$ has \ntrank\ less than~$\alpha$, and therefore satisfies the induction hypothesis. Hence we find an end-faithful spanning tree $T_C$ of $G[C\cup P]$.
By Lemma~\ref{lemma: endfaithful and complete countable set} we may assume that the path $P$ is a subgraph of $T_C$ if this path is a ray. 
It is straightforward to check that $H_C:= T_C-X$ is as desired.
\end{proof}

\section{Rayless spanning trees}\label{section: rayless spanning trees}

\noindent In this section we prove that for every \ntr\ graph $G$, having a rayless spanning tree is equivalent to all the ends of $G$ being dominated.
Our proof builds on the following theorem from the third paper of the star-comb series~\cite{StarComb1StarsAndCombs,StarComb2TheDominatedComb,StarComb3TheUndominatedComb,StarComb4TheUndominatingStar} that hides a six page argument:

\begin{theorem}[{\cite[Theorem~1]{StarComb3TheUndominatedComb}}]\label{thm:raylessSTinNSG}
Let $G$ be any graph and let $U\subset V(G)$ be normally spanned.
Then there is a rayless tree $T\subset G$ that includes~$U$ if and only if all the ends of $G$ in the closure of $U$ are dominated~in~$G$.
\end{theorem}

\begin{customthm}{2}\label{thm: normally traceable rayless}
For every \ntr\ graph $G$, having a rayless spanning tree is equivalent to all the ends of $G$ being dominated.
\end{customthm}

\begin{proof}
Let $G$ be any \ntr\ graph.
The forward implication is clear.
By Theorem~\ref{main result: rayless tree} we may prove the backward implication via induction on the \ntrank\ of~$G$. 
For this, we suppose that every end of $G$ is dominated.
If $G$ has \ntrank\ $0$, then it is normally spanned. 
Thus, by Theorem~\ref{thm:raylessSTinNSG}, the graph $G$ has a rayless spanning tree.
For the induction step suppose that $G$ has \ntrank\ $\alpha>0$, and let $X\subset V(G)$ be any vertex set that is normally spanned in $G$ and satisfies that every component of $G-X$ has \ntrank\ less than~$\alpha$ in $G$.
By replacing $X$ with any normal tree in $G$ that contains $X$, we may assume that $X$ is the vertex set of a normal tree $\nt\subset G$; indeed, every component of $G-X$ still has \ntrank\ less than~$\alpha$ in $G$ by Lemma~\ref{lemma:ntrankOfSubgraphs}~(i). 

We claim that it suffices to find in every component $C$ of $G-X$ a rayless spanning forest $H_C$ such that every component of $H_C$ sends an edge in $G$ to~$X$.
This can be seen as follows.
Suppose that we find such a rayless spanning forest $H_C$ in every component $C$ of $G-X$.
By Theorem~\ref{thm:raylessSTinNSG} we find a rayless tree $\rlt\subset G$ that contains $X=V(\nt)$.
Then we set $H_D':=H_C\cap D$ for every component $D$ of $G-\rlt$ and the component $C$ of $G-X$ containing it.
Now consider the spanning forest $H$ of $G$ that is the union of all forests $H_D'$ with the tree~$\rlt$.
Then a rayless spanning tree of $G$ arises from $H$ by Lemma~\ref{lemma: reflecting forest}.

To complete the proof, we show that every component $C$ of $G-X$ has a rayless spanning forest $H_C$.
So let $C$ be any component of $G-X$.
If the neighbourhood $N(C)\subset\nt$ is finite, then we let $P$ be the path in $\nt$ formed by the down-closure of $N(C)$ in~$\nt$.
Otherwise, we let $P$ be the union of the ray $R$ formed by the down-closure of $N(C)$ in~$\nt$ with a star in $G$ \at~$R$.
We claim that every end of the graph $G[C\cup P]$ is dominated.
This is clear if $N(C)$ is finite, hence we may assume that $P$ is the union of $R$ with a star attached to~$R$.
Now let $S$ be any ray in $G[C\cup P]$.
By the choice of~$P$, we may assume that $S$ is inequivalent to~$R$ in $G[C\cup P]$.
Hence there exists a finite set $Y$ of vertices of $G[C\cup P]$ that separates $S$ and~$R$.
Since $R$ contains the entire neighbourhood of $C$ in~$G$, the two rays $S$ and $R$ are also separated by~$Y$ in~$G$.
In particular, the component $K$ of $G-Y$ that contains a tail of~$S$ is included in~$C$.
The ray $S$ is dominated in $G$ by some vertex $d$, hence $d$ is the centre vertex of a star in~$G$ attached to~$S$.
Without loss of generality, the paths that form this star have no inner vertices in the finite vertex set~$Y$.
Then this star is contained in $G[K\cup Y]\subset G[C\cup P]$ as desired.
Thus, by
Lemma~\ref{lemma: IH}, each $G[C\cup P]$ satisfies the induction hypothesis. 
Hence we find a rayless spanning tree $T_C$ of $G[C\cup P]$, and
$H_C:= T_C-X$ is as desired.
\end{proof}

\section*{Outlook}

\noindent 
Every complete graph has both an end-faithful spanning tree and a rayless spanning tree.
All the countable complete graphs are \ntr , but we claim that no uncountable complete graph is \ntr . 
Otherwise it would have a \ntrank\ $\alpha$ by Theorem~\ref{main result: rayless tree}.
A set of vertices in a complete graph is normally spanned if and only if it is countable.
Hence deleting any \ntrank\ reducing normally spanned set of vertices from an uncountable complete graph leaves precisely one component that is a copy of the original uncountable complete graph.
But this component has \ntrank\ $\alpha$ as well, a contradiction.

\begin{problem}
Can the \ntrank\ be generalised so that every connected graph has an end-faithful spanning tree if and only if it has a generalised \ntrank ?
\end{problem}

\begin{problem}
Can the \ntrank\ be generalised so that every connected graph has a rayless spanning tree if and only if all its ends are dominated and it has a generalised \ntrank ?
\end{problem}

\bibliographystyle{amsplain}
\bibliography{BIB}
\end{document}